\documentclass[11pt]{article}

\usepackage{graphicx}
\usepackage{amssymb}
\usepackage{amsmath}
\usepackage{amsthm}
\usepackage{latexsym} 
\usepackage{color}
\usepackage{pstricks} 
\usepackage{psfrag}
\usepackage{rotating} 

\newtheorem{theorem}{Theorem}[section]
\newtheorem{lemma}[theorem]{Lemma}

\newtheorem{conjecture}[theorem]{Conjecture}
\newtheorem{proposition}[theorem]{Proposition}

\newtheorem{observation}[theorem]{Observation} 

\theoremstyle{definition}
\newtheorem{definition}[theorem]{Definition}
\newtheorem{example}[theorem]{Example}

\usepackage{hyperref}

\newcommand\HHH{\mathbb{H}^3}

\newcommand\Isom{\text{Isom}}
\newcommand\homeo{\cong} 
\newcommand\lb{\lbrace} 
\newcommand\rb{\rbrace} 

\title{Boundaries of Kleinian groups}
\author{Peter Ha\"issinsky, Luisa Paoluzzi, and Genevieve Walsh} 

\begin{document} 
\maketitle

\begin{abstract}We review the theory of splittings of hyperbolic groups, as determined by the topology of the boundary.  We give explicit examples of certain phenomena and then use this to describe limit sets of Kleinian groups up to homeomorphism.

\end{abstract} 

\let\thefootnote\relax\footnote{We would like to thank CIRM (Centre International de Recontres Math\'ematiques) for funding the 2-week ``Research in Pairs" program where this project was initiated.  P. Ha\"isinsky and L. Paoluzzi were partially supported by ANR project  GDSous/GSG no. 12-BS01-0003-01. G. Walsh was partially supported by NSF grant 1207644. } 
\section{Introduction and Background}

A {\it Kleinian group}  $G$ is a discrete subgroup of $\mathbb{P} SL_2(\mathbb{C})$. It acts properly discontinuously 
on the hyperbolic $3$-space $\HHH$ via orientation-preserving isometries and it acts  on the Riemann
sphere $\widehat{\mathbb{C}}$ via M\"obius transformations. The latter action is usually not properly discontinuous: there is 
a canonical and invariant partition $$\widehat{\mathbb{C}} = \Omega_G \sqcup \Lambda_G$$
where $\Omega_G$ denotes the {\it ordinary set}, which is the largest  open set of $\widehat{\mathbb{C}}$ on which $G$ acts properly discontinuously, 
and where $\Lambda_G$ denotes the  {\it limit
set}, which is the minimal $G$-invariant compact subset of $\widehat{\mathbb{C}}$. 
A Kleinian group $G$ is {\it convex-cocompact} if there is a convex
subset $\mathcal{C}$ invariant under $G$ such that the restriction of the action of $G$ to $\mathcal{C}$ is cocompact.

We wish to address the following general problem.

\bigskip

\noindent{\bf Problem. --- } {\it Classify (convex-cocompact) Kleinian groups from the topology of their limit set. 
In particular, what can be said of two convex-cocompact Kleinian groups which have homeomorphic limit sets?}

\bigskip

We will provide some examples showing that there are no obvious answers and that certainly an interesting answer might
depend on the specific topology of the limit sets we are considering.  We observe below in Section 3.1 that a particular type (graph-Kleinian) of Kleinian group can be detected via the boundary and we expect that a different type (mixed Kleinian surface) is not detected via the boundary. This question may be posed more generally for word hyperbolic
groups: we will start reviewing their definition and main properties with respect to this problem and similar ones.

\subsection{Hyperbolic spaces and groups in the sense of Gromov}
Mirroring a key property of $\mathbb{H}^n$, a geodesic metric space $X$ is called {\it Gromov-hyperbolic} or {\it $\delta$-hyperbolic} if there exists a $\delta \geq 0$ such that for every geodesic triangle in $X$, the third side is contained the union of the $\delta$-neighborhoods of the other two.  Analogous  to Fuchsian and Kleinian groups, {\it hyperbolic groups} are groups $G$ that by definition act {\it geometrically} (properly discontinuously, co-compactly and by isometries) on some Gromov-hyperbolic metric space $X$.  Since any two geodesic metric spaces on which $G$ acts geometrically are quasi-isometric \cite[p 141]{BH} and any geodesic metric space that is quasi-isometric to a hyperbolic space is hyperbolic \cite[p 409-410]{BH}, this notion is well-defined.  We will abuse notation and say that $G$ is quasi-isometric to a space $X$.  For example, if $G$ is the fundamental group of a surface of genus $g \geq 2$,  then $G$ is hyperbolic and quasi-isometric to $\mathbb{H}^2$.

When $X$ is a hyperbolic metric space, we can define the boundary of $X$, $\partial X$, which is a topological space.  An important property of this boundary is that when $X$ is hyperbolic, and $X'$ is quasi-isometric to $X$, the boundaries $\partial X$ and $\partial X'$ are homeomorphic. Therefore, we consider this to be the boundary of $G$ when $G$ acts geometrically on $X$.  We denote this boundary $\partial G  \homeo \partial X$, when $X$ is  any geodesic space on which $G$ acts geometrically. 

There are many equivalent definitions of the boundary of a hyperbolic space $X$.  For our purposes we will consider the geodesic boundary. The points in this space consist of equivalence classes of geodesic rays in $X$, where two rays are equivalent if they have bounded Hausdorff distance.  We can define a natural topology on the boundary given by the following neighborhood basis. For   $p \in \partial X$, 
$$V(p,r) = \lbrace q \in \partial X: \text{for some geodesic rays $\gamma_1$ }$$
$$\text{and $\gamma_2$}, [\gamma_1] = p, [\gamma_2] =q, \liminf_{s,t\rightarrow \infty}(\gamma_1(s), \gamma_2(t))_x \geq r \rbrace$$ 
where $(y,z)_x = \frac{1}{2}(d(x,y) + d(x,z) -d (y,z))$.   Then as $r \rightarrow \infty$, these nested sets form a basis for the topology of $\partial X$. 

Since the topological type of the boundary is an invariant of the quasi-isometry class of a hyperbolic group, it is natural to ask if hyperbolic groups with homeomorphic boundaries are quasi-isometric.  This is false.  There are examples of hyperbolic buildings given by Bourdon \cite{Bourdonbuildings} 
which are not quasi-isometric but which have homeomorphic boundaries.  
Their associated isometry groups act geometrically on the buildings giving counter-examples for groups. 
It should be noted that the boundary carries a canonical quasiconformal structure which determines the group up to quasi-isometry \cite{Paulin}. 
A natural  numerical quasi-isometry invariant coming from the quasiconformal structure is the conformal dimension of the boundary \cite{pansu:confdim}.
In Bourdon's examples, this invariant suffices to distinguish buildings from one another. We will give examples of groups  in the sequel
which all have the same conformal dimension, equal to one. 

\subsection{Convex co-compact Kleinian groups} 
An important class of hyperbolic groups are convex co-compact Kleinian groups.  A Kleinian group is a discrete subgroup of $\Isom(\HHH)$.  When the quotient of $\mathbb{H}^3$ by a Kleinian group $G$ is a manifold $\mathbb{H}^3 / G$, we call it a {\it Kleinian manifold}. Regarding $\mathbb{H}^3$ in the Poincar\'e ball model, $\partial \HHH   \homeo S^2$, and $\HHH \cup \partial \HHH$ is naturally a closed ball.  Isometries of $\HHH$ extend to (and are extensions of) conformal transformations of $S^2$.  If $G$ is a Kleinian group, the closure in $S^2$ of an orbit of a point doesn't depend on the point.  This set in $S^2$ is called the limit set of $G$, $\Lambda(G)$.  Let $C(G)$ denote the convex hull of $\Lambda(G)$ in $\HHH \cup \partial \HHH$. A Kleinian group $G$ leaves its limit set invariant, and therefore acts properly discontinuously and by isometries on $C(G) \setminus \Lambda(G)$.   Since $\HHH$ is a $\delta$-hyperbolic space, so is any convex subset of $\HHH$.  When $G$ acts co-compactly on $C(G) \setminus \Lambda(G)$, $G$ is a hyperbolic group and  $\partial G \homeo \Lambda(G)$.  In this case we say $G$ is a {\it convex co-compact Kleinian group}.    For example, the fundamental groups of closed hyperbolic 3-orbifolds, the fundamental groups of hyperbolic orbifolds with totally geodesic boundary, free groups, and the fundamental groups of closed surfaces of genus greater than 1 all have actions as convex co-compact Kleinian groups.  The fundamental groups of hyperbolic knot complements do not, since they are not hyperbolic groups.   When $G$ is a convex co-compact Kleinian group and $(C(G) \setminus \Lambda(G) ) / G$ is a manifold, we call the quotient a {\it convex co-compact Kleinian manifold}. 

The class of convex co-compact Kleinian groups exhibits both rigidity and non-rigidity with respect to boundary homeomorphism.  For example, any hyperbolic group with boundary homeomorphic to two points is virtually cyclic and any hyperbolic group with boundary a Cantor set is virtually free \cite[\S\,8]{KB}.  Due to the work of many authors \cite{G, CJ, Bow} any hyperbolic group with boundary homeomorphic to $S^1$ is virtually a Fuchsian group.  Thus for groups with these boundaries,  $\partial G \homeo \partial G'$  implies that $G$ is quasi-isometric to $G'$.  

On the other hand, the situation for Kleinian groups with Sierpi\'nski carpet boundary is somewhat different.  The Sierpi\'nski carpet is the unique planar, 1-dimensional, connected, locally connected, compact  topological space without cut points or local cut points.  The Sierpi\'nski carpet can be realized as the complement of a union of round open discs in $S^2$.  When $\Gamma$ is the fundamental group of a hyperbolic 3-manifold with totally geodesic boundary, $\Gamma$ can be realized as a convex co-compact Kleinian group $G$ where $\Lambda(G)$ (and hence $\partial G$) is homeomorphic to the Sierpi\'nski carpet.  The stabilizers of the discs are precisely the conjugates of the surface groups corresponding to the totally geodesic boundary components.  The boundaries of the discs are the boundaries of these surface subgroups. These 3-manifold groups have homeomorphic boundaries and may not be quasi-isometric.  

Frigerio \cite{Frig} has proven that the fundamental groups of hyperbolic manifolds with totally geodesic boundary are quasi-isometric exactly when the groups are commensurable.  It is easy to construct non-commensurable hyperbolic manifolds with totally geodesic boundary: Every manifold $M$ with 1 totally geodesic boundary component contains a knot $k$ such that $M \setminus K$ is a hyperbolic manifold with 1 cusp and 1 totally geodesic boundary component, by Myers \cite{Myers}.  Then high-enough Dehn filling on this cusp will produce hyperbolic manifolds whose volumes approach the volume of the cusped manifold. Since their boundaries have the same genus, any commensurable manifolds in this set must have the same volume.  As Dehn filling produces manifolds with infinitely many different volumes, there are infinitely many commensurability classes in this set. Then by Frigerio, there are infinitely many quasi-isometry classes in this set.   Bourdon and Kleiner \cite{BK} give a completely different example of infinitely many groups with Sierpi\'nski carpet boundaries where the groups have different conformal dimensions which shows they are not quasi-isometric.

A less restrictive rigidity question is: If $\partial G \homeo \partial G'$, does $G$ act geometrically on the same type of space as $G'$? (Here type might mean different things.) In the category of convex co-compact Kleinian groups, there are several important open questions in this spirt. 

\begin{conjecture} [Cannon] \cite{Cannonconj}  If $G$ is a hyperbolic group with $\partial G \homeo S^2$, $G$ acts geometrically on $\mathbb{H}^3$. 
\end{conjecture} 

\begin{conjecture} [Kapovich and Kleiner] \cite{KK}  If $G$ is a hyperbolic group with $\partial G \homeo S$, where $S$ is the Sierpi\'nski carpet, then $G$ acts geometrically on a convex subset of $\mathbb{H}^3$ bounded by totally geodesic planes. 
\end{conjecture}

\section{Review of Bowditch's splittings} {\label{splittings}}

We now briefly review Bowditch's theory of an important relationship between the topology of the boundary of a group and the group's algebraic structure.  We remark that there are prior and related theories of group splittings which generalize the splittings of a manifold along essential spheres, tori, and cylinders (\cite{ScottSwarup}). 

\begin{theorem}[Theorem 0.1, \cite{Bowcut}]\label{thm:jsj}  Suppose that $\Gamma$ is a one-ended hyperbolic group, which is not a co-compact Fuchsian group. Then there is a canonical splitting of $\Gamma$ as a finite graph of groups such that each edge group is two-ended, and each vertex group is of one of the following three types:
\begin{enumerate} 
\item  a two-ended subgroup,
\item a maximal ``hanging fuchsian" subgroup, or
\item a non-elementary quasiconvex subgroup not of type (2).
\end{enumerate} 
These types are mutually exclusive, and no two vertices of the same type are adjacent.
Every vertex group is a full quasiconvex subgroup. Moreover, the edge groups that connect to any given vertex group of type (2) are precisely the peripheral subgroups of that group. \end{theorem} 
Note: the peripheral subgroups are the (orbi-)boundaries of the surfaces (orbifolds) in type (2) 

The graph of group splitting can be seen from the topology of the boundary as follows.  The local cut points of the boundary $M = \partial G$ are the points $p \in M$ such that $M \setminus p$ has more than one end. The number of ends is called the valency of $p$, $Val(p)$.  Two local cut points $x$ and $y$ are said to be equivalent if $x=y$ or if $Val(x) = Val(y)=n$ where $M \setminus \lb x,y \rb$ has $n$ components.  Bowditch proves that all equivalence classes of local cut points contain either 2 or infinitely many elements.  The stabilizers of pairs of equivalent local cut points are 2-ended groups and hence virtually cyclic.  Such a pair $\lb x,y \rb$ will correspond to a splitting of the group where the vertex groups are the stabilizers of the the closures of the components of $M \setminus \lb x,y \rb$,  and another copy of $\mathbb{Z}$ for the pair $\lb x,y \rb$. 
The infinite equivalence classes are all valence 2, and each infinite equivalence class $\sigma$ comes equipped with a cyclic order.   The closure $\bar \sigma$ is a cyclically ordered Cantor set,  whose stabilizer is a free group subgroup of $G$.  The equivalent pairs of points in $\bar \sigma \setminus \sigma$  correspond to 2-ended groups, which are edge groups in the graph of groups splitting.  The stabilizer of $\bar \sigma$ is a virtually a surface group which is ``hanging onto" the rest of the group via its boundary subgroups.   
The rigid pieces will arise as follows.  Each splitting along a two-ended group will result in several components, where the stabilizers of the closures of these components are subgroups of $G$.  The intersections of these subgroups (corresponding to intersections of the components) may be a hanging Fuchsian group as above, which can be detected via the local cut points.   On the other hand, they could have no further splittings, in which case these group are the rigid subgroups. We give such an example below.

\section{Examples}  
We consider examples of one-ended convex-cocompact Kleinian groups where each piece in its JSJ decomposition is a free group.
It follows that their conformal dimensions are all equal to one so conformal dimension cannot be used to distinguish quasi-isometry classes \cite{carrasco:confgauge}. 

\subsection{Books of $I$-bundles} 

A {\it graph-Kleinian manifold}  is a convex co-compact Kleinian manifold that admits a decomposition along essential annuli so that each piece is a solid torus or an $I$-bundle over a compact surface with boundary.  For each $I$-bundle over $S$, we require that $\partial S \times I$ is exactly its intersection with the annuli.  In terms of the Bowditch decomposition above, the vertex groups are all of type (1) or (2).  A graph-Kleinian manifold is exactly a Kleinian manifold with boundary whose double is a graph-manifold.  These are often called ``Books of $I$-bundles" as in \cite[p 286]{cullershaleninvent1994}.

\begin{observation} Let $G$ be a hyperbolic group with boundary homeomorphic to the boundary of a graph-Kleinian group.  Then $G$ is virtually a graph-Kleinian group. \label{obs} \end{observation}

The proof is a special case of \cite[Theorem 1.2]{PeterABC}, although the statement is slightly different.   We give a brief outline of the proof in our situation.  Let $G$ be a group whose boundary is homeomorphic to the boundary of a graph-Kleinian group $K$.  By Bowditch's JSJ splitting theorem above, $G$ is a amalgamated product of virtually Fuchsian groups, the vertex groups,  amalgamated over edge groups which are their boundary subgroups.  Since this  decomposition forms a malnormal quasi-convex hierarchy, the group is virtually special, and hence quasi-convex subgroups of $G$  are separable by \cite{wise:qcvxh}.  Therefore, there is a finite index subgroup $G'$ of $G$ with a graph of groups decomposition where the vertex groups are torsion free Fuchsian groups and the edge groups are cyclic peripheral subgroups.  Furthermore in this decomposition of $G'$ the generators of the peripheral subgroup are primitive in $G'$.  Therefore, we can form a 3-manifold with fundamental group $G'$ which is a union of thickened surfaces $S \times I$ glued along $\partial S \times I$ to solid tori.  The gluing pattern is determined by the graph of groups decomposition given by Bowditch. Then applying Thurston's Uniformization Theorem, $G'$ can be realized as a convex co-compact Kleinian group.  Finally, Bowditch's characterization of the hanging Fuchsian groups implies that $G'$ is a graph-Kleinian group.

\subsection{One rigid piece} 

Now we consider the case when the Kleinian manifold is made out of surfaces, but some of the pieces may be rigid.  A {\it mixed Kleinian surface group} is a one-ended convex co-compact Kleinian group $\Gamma$ such that the Kleinian manifold $C(\Gamma)/\Gamma$  admits a decomposition along essential annuli so that each piece is a solid torus or an $I$-bundle over a compact surface with boundary (so $\partial S \times I$ is contained in the set of annuli). Furthermore, there must be at least one rigid piece in the Bowditch decomposition of $\Gamma$.  In contrast with graph-Kleinian groups above, it is not at all clear that an analog of Observation \ref{obs} holds for mixed Kleinian surface groups. We give a simple example, which seems to be very special. 

\begin{example}(The ABC example)  Consider a genus two handlebody $H$ with standard generators $a$ and $b$.  Let $c = [a,b]$.  Note that $H$ is homeomorphic to  $ T \times [0,1]$, where $T$ is a one-holed torus.  Now embed curves representing $a$, $b$ and $c$ as follows.  The curve $a$ lies in $T \times 0$, $b$ lies in $T \times 1$, and $c$ is the boundary curve on $T \times \frac{1}{2}$. See Figure \ref{twoways}. Now attach 3 copies of $S \times I$, where $S$ is a one-holed surface, to regular annular neighborhoods of $a$, $b$ and $c$ along $\partial S \times I$.  We denote the manifold by $M_{abc}$ and its fundamental group by  $\Gamma_{abc}$.  \end{example} 
 
 \begin{figure}[h]

\begin{center}
 {
 \psfrag{a}{$a$}
 \psfrag{b}{$b$}
 \psfrag{c}{$c$}
  \psfrag{A}{$a$}
 \psfrag{B}{$b$}
 \psfrag{C}{$c$}
  \includegraphics[height=3.5cm]{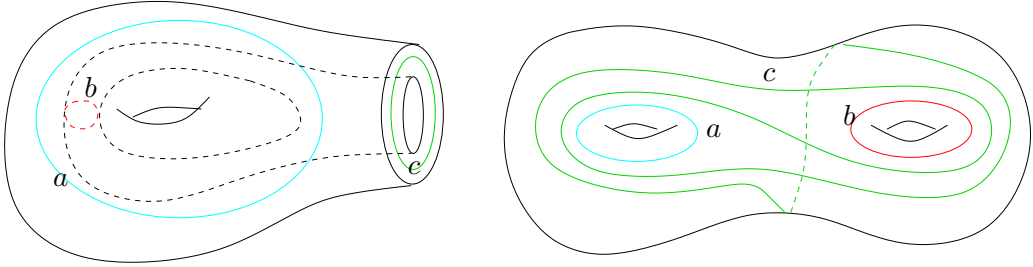}
 }
\end{center} 
\caption{The curves $a$, $b$ and $c$, on $T \times [0,1]$ and on $H$, respectively   }
\label{twoways}
\end{figure}

\begin{proposition} The group $\Gamma_{abc}$ is a mixed Kleinian surface group.  \end{proposition} 
\begin{proof} 
The group can be realized as a convex co-compact Kleinian group by applying 
Thurston's uniformization theorem to $M_{abc}$.  Since the boundary of $M_{abc}$ is incompressible, $\Gamma_{abc}$ is one-ended.   The Bowditch decomposition will have cyclic  vertex groups corresponding to regular solid toric neighborhoods of the three attaching curves $a$, $b$ and $c$. There will be three ``maximally hanging Fuchsian" vertex groups corresponding to the three copies of $S \times I$. 

It remains to show that the manifold $T \times I$,  pared along annular neighborhoods of $a$, $b$ and $c$ (denoted $N_a$, $N_b$, $N_c$) is acylindrical, and thus rigid. Indeed, consider $T \times I$, doubled along $\partial(T \times I) \setminus (N_a \cup N_b \cup N_c)$.  We claim this double $W$ is homeomorphic to 
the exterior ${\mathcal B}$ of the Borromean rings in $S^3$. It is well-known
that the complement of the Borromean rings admits a hyperbolic metric of finite 
volume and this implies that $W$ is atoroidal and the pared manifold is 
acylindrical.  That $W$ is homeomorphic to $\mathcal{B}$ is well-known and was first noticed in Hodgson's Thesis \cite[Section 4]{Hodgson}.  

 In Figure \ref{brings}, opposite sides of the cube should be identified to obtain the double, so that $W$ can be seen also as the exterior of a link in the $3$-torus. 
A standard computation shows that the fundamental group of $W$ admits the 
following presentation: $\langle a,b,c \ | \ [a,[B,c]], [b,[C,a]], [c,[A,b]] \rangle$ 
(where capital letters denote inverses). 
It is easy to see that this group is isomorphic to the fundamental group of 
${\mathcal B}$. Note that $W$ has a cover $\tilde W$ corresponding to the 
universal cover of the $3$-torus: this is ${\mathbb R}^3$ minus the
tubular neighborhoods of infinitely many lines parallel to the standard axes.  This is also a cover of the Borromean rings exterior. Indeed, the Borromean rings' exterior is obtained as the quotient of this cover by the action of the group generated by $\pi$-rotations of ${\mathbb R}^3$ about the lines whose neighborhoods have been removed.  Note also that a fundamental domain for $W$ is compact.  Therefore, the complement of the link in $T^3$ in Figure \ref{brings} admits a complete finite volume hyperbolic metric, induced by that of the cover of the Borromean rings complement.  By Mostow rigidity, using the fact that $\pi_1(W)=\pi_1({\mathcal B})$, $W$ is homeomorphic to the Borromean rings exterior.  

Therefore, $\Gamma_{abc}$ has the property that all the vertex groups in the graph of groups decomposition are (free) surface groups, and at least one of them is a rigid group. \end{proof} 

Note that, since the vertex group corresponding to $H$ is rigid, there is a hyperbolic structure on a genus 2 handlebody, where the curves $a$, $b$, and $c$ are parabolic. After the curves are pinched to parabolics, there are two three-cusped spheres as the boundary, and these are totally geodesic.  
We denote this Kleinian group (which is unique up to conjugation) by $H_W$, as it is the fundamental group of a handlebody where some of the boundary curves are pinched to parabolics. 

\begin{lemma} \label{ap} The limit set of $H_W$, $\Lambda(H_W)$, is the Apollonian gasket. \end{lemma}   
\begin{proof}  The Kleinian group $H_W$ is the fundamental group of the genus two handlebody, pared along the generators $a$ and $b$ and their commutator $c$. This admits a hyperbolic structure with totally geodesic boundary.  The group $H_W$ is a free group on two generators $a$ and $b$ where $a,b$ and the commutator $ABab$ are all parabolic. (Capital letters denote inverses.) Up to conjugacy, we may assume $a: z \rightarrow z+1$ and the fixed point of $b$ is zero.  Solving the equations  so that the commutator is parabolic we obtain: $b: z \rightarrow 1/(2i z +1)$. Then the fixed point of $ABab$ is $(-1 +i)/2$.  The limit set is pictured in Figure \ref{Ap}, and can easily be seen to be conjugate to the Apollonian gasket in Figure 7.5 of \cite{Indra}, which is also the limit set of a free group of rank 2. \end{proof}

\begin{figure} 
\begin{sideways}
\includegraphics[trim={5cm 0 0 0},clip=true, angle=270]{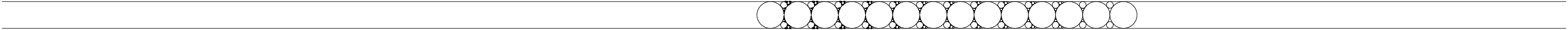} 
\end{sideways} 
\caption{Part of the limit set of $H_W$.    This picture was made with Curt McMullen's LIM program, available at \url{http://www.math.harvard.edu/~ctm/programs/index.html} \label{Ap} }
\end{figure}

\begin{theorem} Any hyperbolic group $G$ with $\partial G$ homeomorphic to $\partial \Gamma_{abc}$ is a virtually mixed Kleinian surface group.  Furthermore, each rigid vertex group in the Bowditch decomposition of $G$ is commensurable with $H_W$. \end{theorem} 

\begin{proof} 

Since $\partial G$ is homeomorphic to $\partial \Gamma_{abc}$, $G$ is one-ended and $\partial G$ has the same Bowditch decomposition as $\partial \Gamma_{abc}$.  In particular, $G$ has a graph of groups decomposition where the edge groups are virtually cyclic, and the vertex groups are either maximally hanging Fuchsian or rigid.  Let $V_R$ be an arbitrary rigid vertex group of $G$.  Then $V_R$ is virtually free, since its boundary is homeomorphic to the boundary of the rigid vertex group of $\Gamma_{abc}$, which is a Cantor set. We continue to denote the finite index free subgroup of $V_R$ by $V_R$. 
The boundary $\partial V_R$  has the property that identifying the endpoints of the edge groups results in a planar set, since this space is homeomorphic to the limit set of $H_W$, $\Lambda(H_W)$.  Then, by Otal \cite{Otal}, $V_R$ is the fundamental group of a handlebody $H$ and there are essential curves $\lbrace c_1,...,c_n \rbrace $ on $H$ such that each edge subgroup incident to $V_R$ is a conjugate of some $c_i$. Then, as $V_R$ is rigid, $V_R$ admits a hyperbolic structure with totally geodesic boundary, where the $c_i$ are parabolic.  Call this structure $H_G$. Then the limit set  $\Lambda(H_G)$ is homeomorphic to $\Lambda(H_W)$.  By Lemma \ref{ap}, $\Lambda(H_G)$ is the Apollonian gasket. Any space consisting of round circles homeomorphic to the Apollonian gasket is conjugate to the Apollonian gasket, by taking the centers of three tangent circles to the centers of three tangent circles.  Then, since $H_W$ and $H_G$ are both finite index in the discrete group of maximal symmetries of the limit set,  $H_W$ is commensurable to $H_G$ which implies $H_G$ is a surface group. As above in Observation 3.1, we have a finite index subgroup $G'$ of $G$ which is a mixed Kleinian surface group. \end{proof} 
\begin{figure}[h]
\label{double} 
\begin{center}
 {
 \psfrag{a}{$a$}
 \psfrag{b}{$b$}
 \psfrag{c}{$c$}
  \includegraphics[height=6cm]{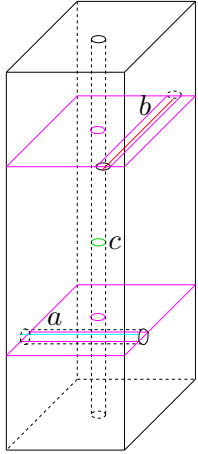}
 }

\end{center}
\caption{A decomposition of the double of the rigid piece. Parallel faces are identified by translations as for the $3$-torus.}
\label{brings}
\end{figure}
\section{Homeomorphism types of limit sets} 

The concept of a tree of metric compacta was introduced by \'Swi\c{a}tkowski  \cite{Swiat}  as a method for understanding more exotic boundaries of hyperbolic groups, but it is also useful in understanding the boundaries of Kleinian groups. 
\begin{definition}(\'Swi\c{a}tkowski) A {\it tree system of metric compacta} is a tuple $\Omega = (T, \lbrace K_t \rbrace, \lbrace \Sigma_e \rb, \lb \phi_e \rb)$ 
such that \begin{enumerate} 
\item $T$ is a countable tree
\item to each $t \in V_{T}$, there is an associated compact metric space $K_t$
\item to each $e \in E_{T}$, there is an associated non-empty compact subset $\Sigma_e \subset K_{\alpha(e)}$, and a homeomorphism $\phi_e: \Sigma_e \rightarrow \Sigma_{\bar e}$ such that  $\phi_{\bar e} = {\phi_e}^{-1}$. 

\item For each $t \in V_T$, the family $\lb \Sigma_e : e$ is an oriented edge emanating from $t \rb$ has the property that for every $\epsilon >0$, only finitely many sets in the family have diameter $>\epsilon$.
\end{enumerate} 
\end{definition} 
The {\em limit of a tree system of metric compacta} is the natural topological object obtained by identifying the vertex spaces via the $\phi_e$ along with the ends of the tree.  The precise definition is given in \cite[Section 1.C]{Swiat}. 
\begin{theorem} Let $M$ be the boundary of a hyperbolic group, and suppose that $M \subset S^2$.  Then $M$ is the limit of a tree system  of metric compacta, where each vertex space is either: 
\begin{enumerate}
\item The 2-sphere $S^2$ (in which case the tree must be trivial) 
\item The Sierpi\'nski carpet.  
\item The circle $S^1$ 
 
\item not connected.  
\item itself the limit of a tree system of metric compacta 

\end{enumerate} 
\end{theorem} 

\begin{proof} Let $M$ be planar with $M \homeo \partial G$.  Then suppose $M$ has no local cut points. If $M$ has covering dimension 2, 
it contains an open subset of $S^2$.  Since $G$ acts minimally on $G$ (there are no non-trivial closed invariant subsets) and $M$ is closed, $M$ has empty frontier $\bar{M} \setminus int(M)$. Thus $M$ must be all of $S^2$.   
Now suppose $M$ is one-dimensional. If it is connected, it is locally connected.  Therefore, in this situation we have a compact, 
connected, locally connected subset of the plane with no local cut point, which must be the Sierpi\'nski carpet \cite{whyburn:sierpinski}.  
If $M$ is connected and 0-dimensional, it is a point.  This cannot happen, since a group cannot act geometrically on a hyperbolic space with one equivalence class of geodesic rays. 

If $M$ is connected and has local cut points, then by Bowditch \cite{Bowcut} $M$ is either homeomorphic to $S^1$ or there is a canonical splitting of $G$ over $\mathbb{Z}$ where $G$ acts on a tree, where the stabilizer of each vertex corresponds to a subgroup $G_v$ of $G$; cf. Theorem \ref{thm:jsj}.  
Then $M$ is a tree of metric compacta, where the compacta are the limit sets of the vertex groups. Similarly, each vertex group has a boundary which is either (2), (3), or (4) above, or has local cut points, and hence a canonical splitting over $\mathbb{Z}$.  We only need to repeat this process finitely many times, by \cite{louder:touikan}.

\end{proof} 

We note that Benoist and Hulin have a characterization of the sets in the Euclidean $S^2$ which are the limit sets of convex cocompact Kleinian groups, \cite{Benoist}.  They are exactly the conformally autosimilar closed subsets of this sphere.

\vskip .3 in
\textsc{Peter Ha\"{\i}ssinsky}

\textsc{Aix-Marseille Universit\'e, CNRS, Centrale Marseille, I2M, UMR 7373,}

\textsc{13453 Marseille, France}

{peter.haissinsky@univ-amu.fr}

\vskip .2 in 
\textsc{Luisa Paoluzzi} 

\textsc{Aix-Marseille Universit\'e, CNRS, Centrale Marseille, I2M, UMR 7373,}

\textsc{13453 Marseille, France}

{luisa.paoluzzi@univ-amu.fr}

\vskip .2  in
\textsc{Genevieve S. Walsh} 

\textsc{Tufts University} 

\textsc{Medford MA 02155}

{genevieve.walsh@gmail.com}

\end{document}